\documentclass[final,notitlepage,12pt,reqno,tbtags]{amsart}
\usepackage{graphicx}
\usepackage{calc}
\usepackage{url}
\usepackage{mathrsfs}
\usepackage{rotating}
\newlength{\depthofsumsign}
\makeatletter
\let\I\@undefined
\makeatother
\usepackage{geometry}
\usepackage{indentfirst}
\usepackage[normalem]{ulem}
\usepackage{float}
\usepackage{amsthm}

\usepackage{enumitem}[2011/09/28]
\setenumerate{align=left, leftmargin=0pt,labelsep=.5em, labelindent=0\parindent,listparindent=\parindent,itemindent=*}
\usepackage{array}
\usepackage{empheq}
\usepackage{natbib}
\usepackage[all]{xy}

\usepackage{multirow}

\newbox\shell
\newcommand{\dia}[2]{\setbox\shell=\hbox{\begin{picture}(180,120)(-90,-60)#1
\put(-90,-60){\makebox(180,120)[b]{\large #2}}\end{picture}}\dimen0=\ht
\shell\multiply\dimen0by7\divide\dimen0by16\raise-\dimen0\box\shell\hfill}

\usepackage{caption}[2012/02/19]

\usepackage{longtable,lscape}

\usepackage{mathptmx}       
\DeclareSymbolFont{operators}{OT1}{txr}{m}{n}
\SetSymbolFont{operators}{bold}{OT1}{txr}{bx}{n}
\def\operator@font{\mathgroup\symoperators}
\DeclareSymbolFont{italic}{OT1}{txr}{m}{it}
\SetSymbolFont{italic}{bold}{OT1}{txr}{bx}{it}
\DeclareSymbolFontAlphabet{\mathrm}{operators}
\DeclareMathAlphabet{\mathbf}{OT1}{txr}{bx}{n}
\DeclareMathAlphabet{\mathit}{OT1}{txr}{m}{it}
\SetMathAlphabet{\mathit}{bold}{OT1}{txr}{bx}{it}
\DeclareSymbolFont{letters}{OML}{txmi}{m}{it}
\SetSymbolFont{letters}{bold}{OML}{txmi}{bx}{it}
\DeclareFontSubstitution{OML}{txmi}{m}{it}
\DeclareSymbolFont{lettersA}{U}{txmia}{m}{it}
\SetSymbolFont{lettersA}{bold}{U}{txmia}{bx}{it}
\DeclareFontSubstitution{U}{txmia}{m}{it}
\DeclareSymbolFontAlphabet{\mathfrak}{lettersA}
\DeclareSymbolFont{symbols}{OMS}{txsy}{m}{n}
\SetSymbolFont{symbols}{bold}{OMS}{txsy}{bx}{n}
\DeclareFontSubstitution{OMS}{txsy}{m}{n}

\usepackage{amssymb,bm,amsmath}
\usepackage{amsfonts}

\usepackage[OT2,T1,T2A]{fontenc}
\usepackage[utf8x]{inputenc}

\usepackage[english,french,german,russian]{babel}
\usepackage{appendix}

\DeclareMathOperator{\D}{d}
\DeclareMathOperator{\I}{Im}
\DeclareMathOperator{\R}{Re}

\DeclareMathOperator{\cov}{cov}

\def\XXint#1#2#3{{\setbox0=\hbox{$#1{#2#3}{\int}$}
     \vcenter{\hbox{$#2#3$}}\kern-.5\wd0}}

\bibpunct{[}{]}{,}{n}{}{;}

\def\qed{\hfill$ \blacksquare$}
\def\eor{\hfill$ \square$}

\theoremstyle{plain}
\newtheorem{theorem}{Theorem}[section]
\newtheorem{proposition}[theorem]{Proposition}
\newtheorem{lemma}[theorem]{Lemma}
\newtheorem{corollary}[theorem]{Corollary}

\newenvironment{remark}[1][Remark]{\begin{trivlist}
\item[\hskip \labelsep {\bfseries #1}]}{\end{trivlist}}

\theoremstyle{definition}

\numberwithin{equation}{section}

\usepackage{tikz}

\tikzset{>=stealth}
\usepackage{pgfplots}

\usepackage{color}

\begin{document}

\selectlanguage{english}
\title{Some geometric relations for equipotential curves}
\author{Yajun Zhou}
\address{Program in Applied and Computational Mathematics (PACM), Princeton University, Princeton, NJ 08544; Academy of Advanced Interdisciplinary Studies (AAIS), Peking University, Beijing 100871, P. R. China }
\email{yajunz@math.princeton.edu, yajun.zhou.1982@pku.edu.cn}

\thanks{\textit{Keywords}: harmonic function, level sets, curvature, diffusion limited aggregation  \\\indent\textit{MSC 2010}: 31A05, 53A04, 82C24
\\\indent * This research was supported in part  by the Applied Mathematics Program within the Department of Energy
(DOE) Office of Advanced Scientific Computing Research (ASCR) as part of the Collaboratory on
Mathematics for Mesoscopic Modeling of Materials (CM4)}
\date{\today}

\maketitle


\vspace{-1.5em}

\begin{abstract}Let $U(\bm r),\bm r\in\Omega\subset \mathbb R^2$ be a harmonic function that solves an exterior Dirichlet problem. If all the level sets of $U(\bm r),\bm r\in\Omega$ are smooth Jordan curves, then there are several geometric inequalities that correlate the  curvature $\kappa(\bm r) $ with the  magnitude of gradient $ |\nabla U(\bm r)|$ on each level set (``equipotential curve''). One of such inequalities is $ \langle [\kappa(\bm r)-\langle\kappa(\bm r)\rangle][|\nabla U(\bm r)|-\langle |\nabla U(\bm r)|\rangle]\rangle\geq0$, where $ \langle \cdot\rangle$ denotes average over a level set, weighted by the arc length of the Jordan curve. We prove such a geometric inequality by constructing  an entropy  for each level set $U(\bm r)=\varphi $, and showing that such an entropy is convex in $\varphi$. The geometric inequality for  $\kappa(\bm r) $  and  $ |\nabla U(\bm r)|$  then follows from  convexity and monotonicity of our entropy formula. A few other geometric relations for equipotential curves are also built on a convexity argument. \end{abstract}

\pagenumbering{roman}


\pagenumbering{arabic}

\section{Introduction}

\subsection{Background and motivations}
Consider a non-constant harmonic function $ U(\bm r)$ that satisfies the Laplace equation\begin{align}\nabla^2 U(\bm r)=0,\quad \bm r\in\Omega\subset \mathbb R^2\label{eq:2D_Laplace}\end{align}in an unbounded  domain $ \Omega$ whose boundary $ \partial \Omega$ is a smooth Jordan curve. We may further impose a  Dirichlet boundary condition that  $ U(\bm r),\bm r\in\partial \Omega$  remains a constant. (By the Riemann mapping theorem in complex analysis, we can deduce from this boundary condition  that all the level sets of $ U(\bm r),\bm r\in\Omega\subset \mathbb R^2$ are  smooth Jordan curves, and $ |\nabla U(\bm r)|\neq0,\bm r\in\Omega\cup\partial \Omega$.) The total flux across the level set  $ \partial \Omega$ is prescribed as\begin{align}
-\oint_{\partial\Omega}\bm n\cdot\nabla U(\bm r)\D s=\Phi>0,\label{eq:2D_flux}
\end{align}  where $\bm n$ denotes outward unit normal vector. (This is also the total flux across every level set of  $ U(\bm r),\bm r\in\Omega\subset \mathbb R^2$, according to the Laplace equation and Green's theorem.) As $ |\bm r|$ goes to infinity, we have the following asymptotic behavior:\begin{align}
 U(\bm r)\sim -\frac{ \Phi}{2\pi}\log\frac{2\pi|\bm r|}{L_{\partial \Omega}},\label{eq:U_inf}
\end{align}   where $ L_{\partial \Omega}=\oint_{\partial\Omega}\D s$ denotes circumference of the boundary.

Such a 2-dimensional exterior Dirichlet problem (``2-exD'' hereafter) is  found in at least three different physical contexts: electrostatic equilibrium  \cite{Jackson:EM}, Hele-Shaw flow \cite{Howison1986,Saffman1986}, and diffusion-limited aggregation \cite{Kakutani1944Brown2D,DLA-PRL,DLA-PRB}  (see Table \ref{tab:2-exD} for details).
 According to \textit{electricians' folklore} (Fig.~\ref{fig:Cassini}), the magnitude $E(\bm r)=|\bm E(\bm r)|$ of the static electric field $\bm E(\bm r)=-\nabla U(\bm r) $ is large around sharp points (with high curvatures) at the air-metal interface $\partial\Omega $. Witten and Sander  \cite[{\S}II.A, p.~27]{DLA-PRB} have  used  this folklore to explain the dendritic morphology in stochastic growth processes, by an analogy between electrostatic  equilibrium and diffusion-limited aggregation.
\begin{table}\caption{Various examples of   exterior Dirichlet problems for harmonic functions in $ \mathbb R^2$\label{tab:2-exD}}\begin{footnotesize}\begin{tabular}{p{.22\textwidth}|p{.22\textwidth}p{.2\textwidth}p{.27\textwidth}}\hline\hline&Electrostatic equilibrium&Hele-Shaw flow&Diffusion-limited aggregation\\\hline Harmonic function $U(\bm r)$  &Electrostatic potential&Fluid pressure &Local density of  particles  \\Region $ \Omega$&Air&Highly viscous phase&Suspension of  particles \\Region $ \mathbb R^2\smallsetminus(\Omega\cup\partial\Omega)$&Metal&Inviscid phase&Particle aggregates \\{Theoretical implications of larger $|\nabla U(\bm r)|,\bm r\in\partial\Omega$}& Higher like\-li\-hood of electrostatic sparks at air-metal interface&Higher local speed of   moving  interface between two fluids&Higher  probability of a random-walking particle hitting a pre-existing aggregate\\{Experimental observations related to larger $ \kappa(\bm r),\bm r\in\partial \Omega$}&Higher tendency to give off  more sparks (``stronger field at sharper points'')  &Higher tendency to stretch out, forming tree-like patterns (``dendritic growth'') &Higher tendency to invite more aggregating particles, forming tree-like patterns (``dendritic growth'')  \\\hline\hline\end{tabular}\end{footnotesize}\end{table}
\begin{figure*}[t]
\begin{minipage}{.68\textwidth}
\begin{center}\begin{tikzpicture}[scale=3]
\draw[thick,fill=gray!25!white] plot[samples=500,domain={-sqrt(1+2^(1/12))}:{sqrt(1+2^(1/12))}] ({\x},{sqrt(-1-\x*\x+sqrt(4*\x*\x+2^(1/6)))})--plot[samples=500,domain={sqrt(1+2^(1/12))}:{-sqrt(1+2^(1/12))}] ({\x},{-sqrt(-1-\x*\x+sqrt(4*\x*\x+2^(1/6)))})--plot[samples=500,domain={-sqrt(1+2^(1/12))}:{sqrt(1+2^(1/12))}] ({\x},{sqrt(-1-\x*\x+sqrt(4*\x*\x+2^(1/6)))});

\foreach \y in {2,3,4,5,6} \draw[gray] plot[samples=500,domain={-sqrt(1+2^(\y/6)/2^(1/12))}:{sqrt(1+2^(\y/6)/2^(1/12))}] ({\x},{sqrt(abs(-1-\x*\x+sqrt(4*\x*\x+2^(\y/3)/2^(1/6))))})--plot[samples=500,domain={sqrt(1+2^(\y/6)/2^(1/12))}:{-sqrt(1+2^(\y/6)/2^(1/12))}] ({\x},{-sqrt(abs(-1-\x*\x+sqrt(4*\x*\x+2^(\y/3)/2^(1/6))))})--plot[samples=500,domain={-sqrt(1+2^(\y/6)/2^(1/12))}:{sqrt(1+2^(\y/6)/2^(1/12))}] ({\x},{sqrt(abs(-1-\x*\x+sqrt(4*\x*\x+2^(\y/3)/2^(1/6))))});

\draw(0,.1)node[below]{$\mathbb R^2\smallsetminus(\Omega\cup\partial\Omega) $};

\end{tikzpicture}\end{center}
\end{minipage}\begin{minipage}{.32\textwidth}\captionsetup{width=\linewidth}\caption{An illustration of electricians' folklore. Some level sets of a 2-dimensional harmonic function  $U(\bm r),\bm r\in\Omega$ are displayed as \textit{gray} curves. The boundary $ \partial \Omega$, drawn in  \textit{black}, is also a level set (``equipotential curve''). Graphically speaking, the values of $ |\nabla U(\bm r)|$ around the sharp tips (which ``stick out'') of level sets are larger (with much denser spacing between level sets)  than those around the depressed pits (which ``cave in'').  \label{fig:Cassini} }\end{minipage}
\end{figure*}
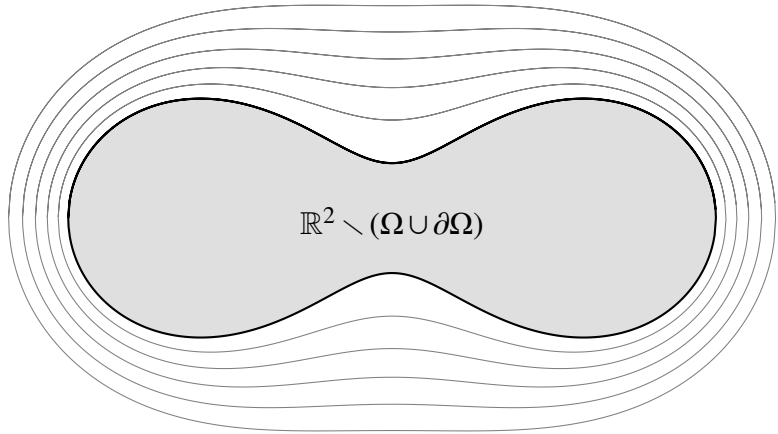

In all these 2-exD problems, the overall shape of a level set affects the gradient  $ \nabla U(\bm r)$ thereupon in a non-local manner: it is technically incorrect to say that a large local curvature causes a large local value of $ |\nabla U(\bm r)| $. Instead of seeking a \textit{pointwise}  \textit{causation}  that ties curvature to gradient, we will look for \textit{statistical correlations} between the curvature of the level set  and the magnitude of   $ \nabla U(\bm r)$, such as  the geometric inequality stated in the abstract.

In the present work and its sequel, the mathematical motivations behind our statistical interpretation of electricians' folklore  are entropy monotonicity techniques developed by Colding \cite{Colding2012} and Colding--Minicozzi \cite{ColdingMinicozzi2013,ColdingMinicozzi2014}.
These techniques not only inject geometric insights into level sets of harmonic
functions, but also add to our geometric understanding of many other types of (linear or non-linear) partial differential equations that arise from classical and quantum physics.
\subsection{Statement of results}
Let $ \kappa(\bm r)$ be the signed curvature of the level set at a point $ \bm r\in\Omega\cup\partial\Omega\subset\mathbb R^2$. In this work, we follow the  convention that the unit circle has positive curvature: $ \kappa=+1$.
For simplicity, we will write $ E(\bm r)=|\nabla U(\bm r)|$. On a level set $ \Sigma$, define the geometric covariance as \begin{align}\cov_{\Sigma}(f_1(\bm r),f_2(\bm r)):=\left\langle \left[f_1(\bm r)-\left\langle f_1(\bm r) \right\rangle_{\Sigma}\right]\left[ f_2(\bm r)-\left\langle f_2(\bm r)\right\rangle_{\Sigma} \right]\right\rangle_{\Sigma}, \end{align}where\begin{align}
\left\langle f(\bm r) \right\rangle_{\Sigma}=\frac{1}{L_\Sigma}\oint_{\Sigma}f(\bm r)\D s=\frac{\oint_{\Sigma}f(\bm r)\D s}{\oint_{\Sigma}\D s}
\end{align}is the average over $\Sigma$, weighted by arc length.

A good half of the main results in the current work are gathered in the theorem below.

\begin{theorem}[Sharp geometric inequalities for 2-exD]\label{thm:2-exD}We have the following inequalities that correlate
$ \kappa(\bm r)$ with $E(\bm r) $  on each level set $ \Sigma$ in 2-exD problems:{\allowdisplaybreaks\begin{align}
\cov_{\Sigma}\left(\frac{\kappa(\bm r)}{E(\bm r)},E(\bm r)\right)\geq{}&0,\label{eq:cov1}\\\left\langle \frac{1}{E(\bm r)} \right\rangle_{\Sigma}\geq{}&\left\langle \frac{\kappa(\bm r)}{E(\bm r)} \right\rangle_{\Sigma}\left\langle\bm n\cdot\bm r\right\rangle_{\Sigma},\label{eq:Longinetti_mono}\\\cov_{\Sigma}(\kappa(\bm r),E(\bm r))\geq{}&0,\label{eq:cov2}\\\cov_{\Sigma}\left( \frac{\kappa(\bm r)}{2\pi},\log \frac{E(\bm r)L_{\Sigma}}{\Phi}\right)\geq{}&\cov_{\Sigma}\left( \frac{E(\bm r)}{\Phi},\log \frac{E(\bm r)L_\Sigma}{\Phi}\right),\label{eq:cov3}\\\left\langle \left[\bm n\times\nabla\frac{\kappa(\bm r)}{E(\bm r)}\right]\cdot\left[\bm n\times\nabla\frac{1}{E(\bm r)} \right] \right\rangle_{\Sigma}\leq{}&0,\label{eq:grad_prod}\\\left\langle \frac{\kappa(\bm r)}{[E(\bm r)]^{2}} \right\rangle_\Sigma\leq{}&\frac{2\pi}{\Phi}\left\langle \frac{1}{E(\bm r)} \right\rangle_\Sigma,\label{eq:Longinetti_cor}
\end{align}}where $ \bm n\times\nabla f(\bm r)$ denotes tangential gradient of a scalar $  f(\bm r)$ (up to a 90$^\circ$ rotation). Furthermore, these inequalities are strict, unless $ \partial\Omega$ is a circle.\end{theorem}

Some of these geometric inequalities (to be proved in \S\ref{sec:ent}) provide quantitative support for the aforementioned electricians' folklore,
and also explain the statistical trend in dendritic growth governed by Hele-Shaw flow or diffusion limited aggregation (see for example, Corollary \ref{cor:Hele-Shaw_mono}).

 Complementary to the 2-exD problems, we may also consider the Green's functions for
2-dimensional interior Dirichlet problems (``2-inD'' hereafter). For an unbounded domain $ \Omega\subset \mathbb R^2$ whose boundary $ \partial\Omega$ is a smooth Jordan curve enclosing the origin $ \mathbf 0$, the 2-inD problem can be formulated as\begin{align}
\left\{\begin{array}{ll}
\nabla^2 G(\bm r)=0, & \bm r\in\mathbb R^2\smallsetminus(\Omega\cup\partial\Omega\cup\{\mathbf 0\}),\\
G(\bm r)=\text{const.,} & \bm r\in \partial\Omega, \\\multicolumn{2}{l}{\displaystyle-\lim_{\varepsilon\to0^+}\oint_{|\bm r|=\varepsilon}\bm n\cdot\nabla G(\bm r)\D s=\Phi,}
\end{array}\right.\label{eq:2-inD_G}
\end{align}where $ \Phi\neq0$. (As before, we automatically have $ |\nabla G(\bm r)|\neq0$ for $\bm r\in\mathbb R^2\smallsetminus(\Omega\cup\{\mathbf 0\})$, by the Riemann mapping theorem.) If we set $ \Phi=1$ in the 2-inD problem, then
$ G(\bm r)=G_D^{\partial \Omega}(\mathbf 0,\bm r)$ is called a  Dirichlet Green's function in electrostatics \cite{Jackson:EM}. These 2-inD problems also occur in the setting  of internal diffusion limited aggregation \cite{MeakinDeutch1986,DiaconisFulton1991}, or the corresponding Hele-Shaw flow as its deterministic limit \cite{LevinePeres2010}.

We have the following interior analog of Theorem \ref{thm:2-exD}.

\begin{theorem}
[Sharp geometric inequalities for 2-inD]\label{thm:2-inD}We have the following inequalities that correlate
$ \kappa(\bm r)$ with $G(\bm r) $  on each level set $ \Sigma$ in 2-inD problems:{\allowdisplaybreaks\begin{align}
\cov_{\Sigma}\left(\frac{\kappa(\bm r)}{|\nabla G(\bm r)|},|\nabla G(\bm r)|\right)\leq{}&0,\tag{\ref{eq:cov1}$'$}\label{eq:cov1'}\\\cov_{\Sigma}\left( \frac{\kappa(\bm r)}{2\pi},\log \frac{|\nabla G(\bm r)|L_{\Sigma}}{\Phi}\right)\leq{}&\cov_{\Sigma}\left( \frac{|\nabla G(\bm r)|}{\Phi},\log \frac{|\nabla G(\bm r)|L_\Sigma}{\Phi}\right),\tag{\ref{eq:cov3}$'$}\label{eq:cov3'}\\\left\langle \left[\bm n\times\nabla\frac{\kappa(\bm r)}{|\nabla G(\bm r)|}\right]\cdot\left[\bm n\times\nabla\frac{1}{|\nabla G(\bm r)|} \right] \right\rangle_{\Sigma}\geq{}&0,\tag{\ref{eq:grad_prod}$'$}\label{eq:grad_prod'}\\\left\langle \frac{\kappa(\bm r)}{|\nabla G(\bm r)|^{2}} \right\rangle_\Sigma\geq{}&\frac{2\pi}{\Phi}\left\langle \frac{1}{|\nabla G(\bm r)|} \right\rangle_\Sigma,\tag{\ref{eq:Longinetti_cor}$'$}\label{eq:Longinetti_cor'}
\end{align}}where the inequalities are strict except when $ \partial \Omega$ is a circle centered at the origin $ \mathbf 0$.\footnote{Unfortunately, due to the reversed inequality in \eqref{eq:cov1'}, we are not yet able to produce an analog of  \eqref{eq:Longinetti_mono} or \eqref{eq:cov2}, for the 2-inD problems. }\end{theorem}

There is a notable reversal of
inequality signs [as compared to \eqref{eq:cov1}, \eqref{eq:cov3}, \eqref{eq:grad_prod} and \eqref{eq:Longinetti_cor} for 2-exD], whose physical significance will be sketched at the end of  \S\ref{sec:Green}.

\subsection{Plan of proof}
In \S\ref{sec:ent}, we study four functionals of $ \kappa(\bm r)$ and $ E(\bm r)$, which are  integrals over the level set $ \Sigma_\varphi:=\{\bm r\in\Omega\cup\partial\Omega|U(\bm r)=\varphi\}$:\begin{align}\mathscr
H(\varphi):={}&\oint_{\Sigma_\varphi}\frac{E(\bm r)}{\Phi}\log \frac{E(\bm r)L_{\Sigma_\varphi}}{\Phi}\D s,\label{eq:Hphi}\\\mathscr E(\varphi):={}&\oint_{\Sigma_\varphi}\frac{\kappa^2(\bm r)}{E(\bm r)}\D s-\oint_{\Sigma_\varphi}\frac{\left\vert\bm n\times\nabla\log E(\bm r)\right\vert^{2}}{E(\bm r)}\D s,\label{eq:Ephi}\\\mathscr F(\varphi):={}&\oint_{\Sigma_\varphi}\frac{\left\vert\bm n\times\nabla\log E(\bm r)\right\vert^{2}}{E(\bm r)}\D s,\label{eq:Wphi}\\\mathscr L(\varphi):={}&\log\oint_{\Sigma_\varphi}\frac{\D s}{E(\bm r)}.\label{eq:Lphi}
\end{align} We show that all these expressions are convex functions in $\varphi$, and their convexity eventually  entails \eqref{eq:cov1}--\eqref{eq:Longinetti_cor}, via  the  first-order derivatives $ \mathscr H'(\varphi)\geq0$, $ \mathscr E'(\varphi)=0$, $\mathscr F'(\varphi)\geq0 $ and $ \mathscr L'(\varphi)+\frac{4\pi}{\Phi}\geq0$. Such convexity arguments are morally similar to entropy monotonicity relations for level sets of Green's functions in manifolds of dimension greater than 2, as developed by Colding \cite{Colding2012} and Colding--Minicozzi \cite{ColdingMinicozzi2013,ColdingMinicozzi2014}.

In \S\ref{sec:Green}, we generalize our analysis to
2-inD problems. While all the functionals in  \eqref{eq:Hphi}--\eqref{eq:Lphi} remain convex in 2-inD,  their limit behavior is different from the 2-exD counterpart. This critical difference causes sign changes in the expressions $ \mathscr H'(\varphi)$, $\mathscr F'(\varphi) $ and $ \mathscr L'(\varphi)+\frac{4\pi}{\Phi}$, which account for the reverse inequalities in Theorem \ref{thm:2-inD}.

In \S\ref{sec:misc}, we show that our derivations are sensitive to space dimension:  there are some analytic and geometric difficulties   that one may encounter  while extending  the present results  to $ \mathbb R^d,d>2$.

\section{Entropy, conservation law and geometric inequalities\label{sec:ent}}
\subsection{Geometric entropy and curvature correlations}\label{subsec:ent_curv}The function $ \mathscr H(\varphi)$  defined in \eqref{eq:Hphi} is the relative entropy between two mutually non-singular probability measures: the flux density $ \D\mu=E(\bm r)\D s /\Phi$ and the line density $ \D\nu=\D s/L_{\Sigma_\varphi}$. By Jensen's inequality, we have \begin{align}
\mathscr H(\varphi)=-\oint_{\Sigma_\varphi}\log\frac{\D\nu}{\D\mu}\D\mu\geq-\log\oint_{\Sigma_\varphi}\frac{\D\nu}{\D\mu}\D\mu=0.
\label{eq:Hnonneg}\end{align}

Before evaluating the derivatives $ \mathscr H'(\varphi)$ and $ \mathscr H''(\varphi)$, we need some geometric preparations.

We assign local orthogonal curvilinear coordinates $\bm r (\varphi,u)$ to points $\bm r\in\Omega\cup\partial\Omega $, so that $ \varphi$ coincides with $ U(\bm r)$, and a pair of  points on different level sets share the same $u$ coordinate if and only if they are joined by an integral curve of $ \nabla U(\bm r) $ (known as ``electric field line'' in electrostatics). In such a  curvilinear coordinate system, we can rewrite the Euclidean metric $ \D \bm r\cdot\D \bm r=(\D x)^2+(\D y)^2$ as \begin{align}
\D \bm r\cdot\D \bm r=\frac{(\D\varphi)^2}{E^{2}}+g(\D u)^2, \quad\text{where }\frac{1}{E}:=\frac{1}{|\nabla U(\bm r)|}=\left\vert \frac{\partial \bm r}{\partial\varphi} \right\vert,g:=\left\vert \frac{\partial \bm r}{\partial u} \right\vert^2.
\end{align}Accordingly, the Laplacian in the Euclidean space $ \Delta\equiv\nabla^2=\frac{\partial^2}{\partial x^2}+\frac{\partial^2}{\partial y^2}$ can be decomposed into \cite[Proposition 1.2]{Zhou2013LevelSet}\begin{align}
\Delta=\Delta_{\Sigma_\varphi}+E^2\frac{\partial^2}{\partial \varphi^2}-\frac{1}{gE}\frac{\partial E}{\partial u}\frac{\partial}{\partial u}.\label{eq:Lap_Lap}
\end{align}Here, \begin{align}
\Delta_{\Sigma_\varphi}=\frac{1}{\sqrt{g}}\frac{\partial}{\partial u}\left( \frac{\sqrt{g}}{g} \frac{\partial}{\partial u}\right)=\frac{\partial^2}{\partial s^2}
\end{align}is the Laplacian on the level set $ \Sigma_\varphi$, whose arc length  parameter  $s$ satisfies $ (\D s)^2=g(\D u)^2$. It is then an elementary exercise in differential geometry to show that \cite[(5), (6) and Proposition 1.3]{Zhou2013LevelSet}\begin{align}
\frac{\partial (E\sqrt{g})}{\partial \varphi}=0,\quad \frac{\partial E}{\partial \varphi}=\kappa,\quad \frac{\partial \kappa}{\partial \varphi}=\frac{\kappa^{2}}{E}+\Delta_{\Sigma_\varphi}\frac{1}{E}.
\label{eq:phi_flow}\end{align}In what follows, we will also frequently need to differentiate $ \D s=\sqrt{g}\D u$ with respect to $\varphi$. Since  $u$ and $ \varphi$ are two independent variables in our  orthogonal curvilinear coordinate system, this boils down to computing the derivative  $ \frac{\partial \sqrt{g}
}{\partial\varphi}=-\frac{\kappa\sqrt{g}}{E}$  via either the standard normal variation formula of volume element in differential geometry, or the first two equalities in \eqref{eq:phi_flow}.
Such a technique also brings us \begin{align}
\frac{\D}{\D\varphi}L_{\Sigma_\varphi}=\frac{\D}{\D\varphi}\oint_{\Sigma_\varphi}\D s=\oint_{\Sigma_\varphi}\frac{\partial \sqrt{g}
}{\partial\varphi}\D u=-\oint_{\Sigma_\varphi}\frac{\kappa\D s
}{E},\label{eq:L'phi}
\end{align}agreeing with Laurence's coarea formula \cite[(77)]{Laurence1989}
up to sign conventions.\begin{proposition}[Convexity of $\mathscr H(\varphi)$]We have $\mathscr H''(\varphi)\geq0 $ in 2-exD, where the equality holds if and only if  $ E$ and $ \kappa$ both remain constant on $ \Sigma_\varphi$.
\end{proposition}\begin{proof}
Combining the geometric preparations above with the Umlaufsatz $ \oint_{\Sigma_\varphi}\kappa\D s=2\pi$, we can readily compute \begin{align}
\mathscr H'(\varphi)=\frac{1}{\Phi}\oint_{\Sigma_\varphi}E\frac{\partial}{\partial\varphi}\left(\log \frac{EL_{\Sigma_\varphi}}{\Phi}\right)\D s=\oint_{\Sigma_\varphi}\frac{\kappa\D s}{\Phi_{}}-\oint_{\Sigma_\varphi}\frac{\kappa\D s}{EL_{\Sigma_\varphi}}=\frac{2\pi}{\Phi}-\oint_{\Sigma_\varphi}\frac{\kappa\D s}{EL_{\Sigma_\varphi}},\label{eq:H'comp}
\end{align}and

\begin{align}\mathscr H''(\varphi)={}&-\oint_{\Sigma_\varphi}\frac{1}{E}\Delta_{\Sigma}\frac{1}{E}\frac{\D s}{L_{\Sigma_\varphi}}-\frac{1}{L_{\Sigma_\varphi}}\oint_{\Sigma_\varphi}\kappa\frac{\partial}{\partial\varphi}\left(\frac{1}{E}\right)\D s-\left(\oint_{\Sigma_\varphi}\frac{\kappa\D s}{E}\right)\frac{\partial}{\partial\varphi}\left(\frac{1}{L_{\Sigma_\varphi}}\right)\notag\\={}&\oint_{\Sigma_\varphi}\left\vert\bm n\times\nabla\frac{1}{E}\right\vert^{2}\frac{\D s}{L_{\Sigma_\varphi}}+\oint_{\Sigma_\varphi}\left(\frac{\kappa}{E}\right)^2\frac{\D s}{L_{\Sigma_\varphi}}-\left(\oint_{\Sigma_\varphi}\frac{\kappa\D s}{EL_{\Sigma_\varphi}}\right)^2\notag\\={}&\left\langle \left\vert\bm n\times\nabla\frac{1}{E}\right\vert^{2}+\left(\frac{\kappa}{E}-\left\langle\frac{\kappa}{E}\right\rangle_{{\Sigma_\varphi}}\right) ^{2}\right\rangle_{{\Sigma_\varphi}}\geq0.\end{align}Here, the term $ \left\vert\bm n\times\nabla\frac{1}{E}\right\vert^{2}\D s=\left(\frac{\partial }{\partial s}\frac{1}{E}\right)^2\D s$ arises from integrating the differential form $ -\frac{1}{E}\Delta_{\Sigma}\frac{1}{E}\D s=-\frac{1}{E}\frac{\partial^2}{\partial s^2}\frac{1}{E}\D s$ by parts.

This proves the convexity of our entropy function $ \mathscr H(\varphi)$. Furthermore, it is clear that we have $ \mathscr H''(\varphi)>0$ unless both $\frac1E $ and $ \frac{\kappa}{E}$  remain constant on $ \Sigma_\varphi$, which is equivalent to our claim in the proposition.\end{proof}

Since we can rewrite the computations above as \begin{align}
\mathscr H'(\varphi)=\frac{2\pi}{\Phi}+\frac{\D}{\D \varphi}\log L_{\Sigma_\varphi},\quad \mathscr H''(\varphi)=\frac{\D^{2}}{\D \varphi^{2}}\log L_{\Sigma_\varphi}\geq0,
\end{align}we also have a by-product stated in the following corollary. \begin{corollary}The circumference $L_{\Sigma_\varphi} $ of the equipotential curve $ \Sigma_\varphi$ is logarithmically convex in $\varphi$. \qed\end{corollary}\begin{remark}This strengthens a previous result of Longinetti \cite[Theorem 3.1]{Longinetti1988}, which demonstrated logarithmic convexity of  $L_{\Sigma_\varphi} $   for convex ring configurations, without using an entropy argument.

We  note that Longinetti's aforementioned result has been generalized by Laurence \cite[Theorem 6]{Laurence1989} to star-convex ring configurations. The work of Laurence \cite{Laurence1989} also includes many higher order analogs of coarea formulae, which are applicable to harmonic functions in any spatial dimensions. I thank an anonymous reviewer for bringing Laurence's work to my attention.
\eor\end{remark}
\begin{proposition}[Sharp inequality for $ \mathscr H'(\varphi)$]In 2-exD, we have $ \mathscr H'(\varphi)\geq0$, which entails  \eqref{eq:cov1}: $\cov_{\Sigma_\varphi}\left( \frac{\kappa}{E},E \right)\geq0$. Both inequalities are strict unless $ \partial\Omega$ is a circle. \end{proposition}
\begin{proof}As $ |\bm r|\to +\infty$, we have $ \varphi\to-\infty$ according
to \eqref{eq:U_inf}, and \begin{align}
\kappa(\bm r)\sim \frac{1}{|\bm r|}, \quad E(\bm r)\sim\frac{\Phi}{2\pi|\bm r|},\label{eq:kappa_E_asympt}
\end{align}so \begin{align}
\lim_{\varphi\to-\infty}\mathscr H'(\varphi)=\frac{2\pi}{\Phi}-\frac{2\pi}{\Phi}\lim_{\varphi\to-\infty}\oint_{\Sigma_\varphi}\frac{\D s}{L_{\Sigma_\varphi}}=0,
\end{align}and \begin{align}
\mathscr H'(\varphi)=\int_{-\infty}^\varphi\mathscr H''(\phi)\D \phi\geq0.\label{eq:H'nonneg}
\end{align}The last inequality can be recast into\begin{align}
\frac{2\pi}{\Phi}-\oint_{\Sigma_\varphi}\frac{\kappa\D s}{EL_{\Sigma_\varphi}}=\frac{L_{\Sigma_\varphi}}{\Phi}\left(\oint_{\Sigma_\varphi}\frac{\kappa E\D s}{EL_{\Sigma_\varphi}}-\oint_{\Sigma_\varphi}\frac{\kappa\D s}{EL_{\Sigma_\varphi}}\oint_{\Sigma_\varphi}\frac{E\D s}{L_{\Sigma_\varphi}}\right)=\frac{L_{\Sigma_\varphi}}{\Phi}\cov_{\Sigma_\varphi}\left( \frac{\kappa}{E},E \right)\geq0,\label{eq:cov1pf}
\end{align} which proves our first geometric inequality stated in \eqref{eq:cov1}.
Such an inequality becomes an equality only if $ \mathscr H''(\phi)=0$ for all $\phi<\varphi$, that is, both $E$ and $\kappa$ remain constant  on every
level set enclosing
$ \Sigma_\varphi$---a scenario that   happens only when $ \partial \Omega$ is a circle.\end{proof}\begin{remark}Following Longinetti \cite{Longinetti1988}, we may also consider a harmonic function defined in an annular domain bounded by two smooth Jordan curves hereafter referred to as inner and outer rings (which are both equipotential curves), along with the Bernoulli boundary condition that $ E(\bm r)=|\nabla U(\bm r)|$ remains constant on the outer ring. By \eqref{eq:H'comp}, we have $ \mathscr H'(\varphi)=0$ on the outer ring. The proof above also brings us  $ \mathscr H'(\varphi)\geq0$ and $\cov_{\Sigma_\varphi}\left( \frac{\kappa}{E},E \right)\geq0$ in the annular domain. The inequalities are strict unless the inner and outer rings are concentric circles.

In fact, most of our results on 2-exD to be developed in \S\ref{subsec:ent_curv}  and \S\ref{subsec:stat_align} (with the only exception of Corollary \ref{cor:Hele-Shaw_mono}) extend naturally to annular domains with the Bernoulli boundary condition on the outer ring. Such extensions are also possible for the results from \S\ref{subsec:cons_law} and  \S\ref{subsec:Lfunc}, so long as some constant terms are adapted to the context of the outer ring. These extensions will generalize Longinetti's work \cite{Longinetti1988} on convex ring domains---annular domains whose inner and outer rings are both convex curves.       \eor\end{remark}
\begin{corollary}[Sharp inequalities related to isoperimetric deficit] For any equipotential curve $\Sigma$ in 2-exD, the geometric inequality     \eqref{eq:Longinetti_mono} $\left\langle \frac{1}{E(\bm r)} \right\rangle_{\Sigma}\geq\left\langle \frac{\kappa(\bm r)}{E(\bm r)} \right\rangle_{\Sigma}\left\langle\bm n\cdot\bm r\right\rangle_{\Sigma}$ holds. Let $A_{\mathfrak D_\varphi}:=\int_{\mathfrak D_{\varphi}}\D^2\bm r
$ be the area of the region $ \mathfrak D_{\varphi}$ bounded by the equipotential curve $ \Sigma_{\varphi}=\partial\mathfrak  D_{\varphi}$, then $ \frac{\D}{\D\varphi}\Big(1-\frac{4\pi A_{\mathfrak D_\varphi}}{L_{\partial\mathfrak D_\varphi}^2}\Big)\geq0$ and $\frac{\D}{\D\varphi}\big(L_{\partial\mathfrak D_\varphi}^{2}-4\pi A_{\mathfrak D_\varphi}^{\vphantom1}\big)\geq0 $ in 2-exD. All these three inequalities are strict unless $ \partial \Omega$ is circular.
\end{corollary}\begin{proof}
We note that $ \frac12\oint_{\Sigma_{\varphi}=\partial\mathfrak D_\varphi}\bm n\cdot\bm r\D  s=\int_{\mathfrak D_{\varphi}}\D^2\bm r=A_{\mathfrak D_\varphi}$. To prove    \eqref{eq:Longinetti_mono} for $ \Sigma=\partial \mathfrak  D$, we compute\begin{align}
L_{\Sigma}\left\langle \frac{1}{E} \right\rangle_{\Sigma}-2\left\langle \frac{\kappa}{E} \right\rangle_{\Sigma}A_{\mathfrak  D}\geq\frac{L_{\Sigma}}{\langle E\rangle_\Sigma}-\frac{2\langle \kappa\rangle_\Sigma}{\langle E\rangle_\Sigma}A_{\mathfrak  D}=\frac{L_{\Sigma}^{2}-4\pi A_{\mathfrak  D}^{\vphantom1}}{\Phi}\geq0.
\end{align} Here, in the first step, we have exploited the Cauchy--Schwarz inequality and the correlation inequality in \eqref{eq:cov1}. In the last step, we have used the isoperimetric inequality.

In view of a special case \cite[\S4]{Talenti1983} of Federer's coarea formula \cite[\S3.2]{Federer1969}\begin{align}
-\frac{\D }{\D\varphi}A_{\mathfrak D_\varphi}=\oint_{\Sigma_{\varphi}}\frac{\D s}{E},\label{eq:coarea}
\end{align}we can rearrange  \eqref{eq:Longinetti_mono}  into
\begin{align}
\frac{4\pi A_{\mathfrak D_\varphi}}{L_{\partial\mathfrak D_\varphi}^2}\left( 2 \left\langle \frac{\kappa}{E} \right\rangle_{\Sigma_\varphi}-\frac{L_{\Sigma}}{A_{\mathfrak  D_{\varphi}}}\left\langle \frac{1}{E} \right\rangle_{\Sigma_{\varphi}}\right)=\frac{\D}{\D\varphi}\frac{4\pi A_{\mathfrak D_\varphi}}{L_{\partial\mathfrak D_\varphi}^2}\leq0.\label{eq:isop_defc_deriv}
\end{align}Therefore,   the left-hand side of the rescaled isoperimetric inequality $ 1-\frac{4\pi A_{\mathfrak D_\varphi}}{L_{\partial\mathfrak D_\varphi}^2}\geq0$ decays monotonically to zero, as $ |\bm r|\to+\infty$, $ \varphi\to-\infty$.

Meanwhile, one can also compute\begin{align}\begin{split}&\frac{\D}{\D\varphi}
\left(L_{\partial\mathfrak D_\varphi}^{2}-4\pi A_{\mathfrak D_\varphi}^{\vphantom1}\right)=2L_{\partial\mathfrak D_\varphi}\frac{\D L_{\partial\mathfrak D_\varphi}}{\D\varphi}-4\pi\frac{\D A_{\mathfrak D_\varphi}
}{\D\varphi}\\={}&-2L_{\partial\mathfrak D_\varphi}\oint_{\Sigma_\varphi}\frac{\kappa\D s
}{E}+4\pi\oint_{\Sigma_{\varphi}}\frac{\D s}{E}\geq-\frac{4\pi L_{\Sigma_{\varphi}}^2}{\Phi}+4\pi\oint_{\Sigma_{\varphi}}\frac{\D s}{E}\\={}&\frac{4\pi L_{\Sigma_{\varphi}}}{\langle E\rangle_{\Sigma_{\varphi}}}\left(\langle E\rangle_{\Sigma_{\varphi}}\left\langle \frac{1}{E} \right\rangle _{\Sigma_\varphi}-1 \right)\geq0,\end{split}
\end{align}where the penultimate inequality comes from \eqref{eq:cov1pf}, and the last step involves the Cauchy--Schwarz inequality. This generalizes Longinetti's proof for the monotonicity of isoperimetric deficit $L_{\partial\mathfrak D_\varphi}^{2}-4\pi A_{\mathfrak D_\varphi} $ in the setting of  convex ring domains with the Bernoulli boundary condition (constant $ E(\bm r)=|\nabla U(\bm r)|$) on the outer ring \cite[Theorem 5.2]{Longinetti1988}.

In other words, as one tracks down the electrostatic potential $\varphi$, the equipotential curve $ \Sigma_{\varphi}$ becomes rounder and rounder \big(according to the decay of    either   $1-\frac{4\pi A_{\mathfrak D_\varphi}}{L_{\partial\mathfrak D_\varphi}^2}$ or $L_{\partial\mathfrak D_\varphi}^{2}-4\pi A_{\mathfrak D_\varphi} $\big). This trend is strictly monotone, unless both $E$ and $\kappa$ remain constant  on every
level set, that is, unless  $ \partial \Omega$ is a circle.
\end{proof}

\begin{proposition}[Non-negative correlation between curvature and field intensity]In 2-exD, we have $ \cov_{\Sigma_\varphi}\left(\kappa, E\right)\geq0$ according to \eqref{eq:cov2}. The inequality is strict unless $\partial \Omega $ is circular.\end{proposition}
\begin{proof}
To prove  \eqref{eq:cov2}, we differentiate\begin{align}
L_{\Sigma_\varphi}\cov_{\Sigma_\varphi}\left(\kappa, E\right)=\oint_{\Sigma_\varphi}\kappa E\D s-\frac{2\pi\Phi}{L_{\Sigma_\varphi}}
\end{align}  as follows:\begin{align}\frac{\D}{\D\varphi}\left[ L_{\Sigma_\varphi}\cov_{\Sigma_\varphi}\left( \kappa, E\right) \right]={}&\oint_{\Sigma_\varphi}E\Delta_{\Sigma_{\varphi}}\frac{1}{E}\D s+\oint_{\Sigma_\varphi}\kappa^{2}\D s-\frac{2\pi\Phi}{L^{2}_{\Sigma_\varphi}}\oint_{\Sigma_\varphi}\frac{\kappa\D s}{E}\notag\\={}&\oint_{\Sigma_\varphi}\left\vert\bm n\times\nabla\log E\right\vert^2\D s+\oint_{\Sigma_\varphi}\kappa^{2}\D s+\frac{2\pi\Phi}{L_{\Sigma_\varphi}}\left[ \mathscr  H'(\varphi)- \frac{2\pi}{\Phi}\right]\notag\\\geq{}& L_{\Sigma_\varphi}\left\langle \big\vert\bm n\times\nabla\log E\big\vert^{2}+\big(\kappa-\left\langle\kappa\right\rangle_{{\Sigma_\varphi}}\big) ^{2}\right\rangle_{{\Sigma_\varphi}}\geq0,\label{eq:d_cov}\end{align}and note that $\lim_{\varphi\to-\infty} L_{\Sigma_\varphi}\cov_{\Sigma_\varphi}\left( \kappa, E\right)=0$. Here, we have exploited \eqref{eq:H'nonneg} in the penultimate step of \eqref{eq:d_cov}. Again, we have a strict  inequality $ \cov_{\Sigma_\varphi}\left(\kappa, E\right)>0$ unless $ \kappa$ and $E$ remain constant on each level set enclosing $ \Sigma_\varphi$, which only occurs when $ \partial \Omega$ is a circle.\end{proof}

The  inequality \eqref{eq:cov2} has a simple application in the Hele-Shaw flow \cite{Howison1986,Saffman1986}, where the time-dependent boundary surface $\partial \Omega_t $ evolves according to $ \frac{\partial \bm r}{\partial t}=v(\bm r,t)\bm n$ for $ \bm r\in\partial \Omega_t$. Here, we have $ v(\bm r,t)=|\nabla U(\bm r,t)|$, where $U(\bm r,t)$ solves 2-exD in the unbounded region $ \Omega_t$ that evolves in time. This moving boundary behavior is different from the electrostatic setting, where $ \frac{\partial \bm r}{\partial\varphi}=-\frac{\bm n}{|\nabla U(\bm r)|}$.
In lieu of \eqref{eq:L'phi}, we have the following normal variation of arc length:\begin{align}
\frac{\D }{\D t}L_{\partial\Omega_t}=\oint_{\partial\Omega_t}\kappa(\bm r,t)v(\bm r,t)\D s.
\end{align}In parallel to the coarea formula \eqref{eq:coarea}, we  have \begin{align}
\frac{\D }{\D t}A_{\mathbb R^2\smallsetminus\Omega_t}=\oint_{\partial\Omega_t}v(\bm r,t)\D s.
\end{align}\begin{corollary}[Monotonicity of isoperimetric deficit under Hele-Shaw flow]\label{cor:Hele-Shaw_mono}Under the Hele-Shaw flow, the isoperimetric deficit  $
L^2_{\partial\Omega_t}-4\pi A_{\mathbb R^2\smallsetminus\Omega_t}$ is non-decreasing with respect to  time $t$. \end{corollary}\begin{proof}
Direct computation reveals that \begin{align}
\frac{\D}{\D t}(L^2_{\partial\Omega_t}-4\pi A_{\mathbb R^2\smallsetminus\Omega_t})=2L_{\partial\Omega_t}^2\cov_{\partial\Omega_t}(\kappa(\bm r,t),v(\bm r,t))\geq0.
\end{align}Therefore, when time elapses, the Hele-Shaw flow drives the boundary away from  roundness, as measured by the monotonically non-decreasing quantity $ L^2_{\partial\Omega_t}-4\pi A_{\mathbb R^2\smallsetminus\Omega_t}$.\end{proof}

 \subsection{Geometric conservation law and correlation comparison\label{subsec:cons_law}}Our next goal is to show that  the quantity  $ \mathscr E(\varphi)$
defined in \eqref{eq:Ephi} is in fact a conservation law:
\begin{align}
\mathscr E(\varphi):={}&\oint_{\Sigma_\varphi}\frac{\kappa^2\D s}{E}-\oint_{\Sigma_\varphi}\frac{\left\vert\bm n\times\nabla\log E\right\vert^{2}\D s}{E}\equiv\frac{4\pi^{2}}{\Phi}.
\end{align}In other words, there is an exact identity with respect to the probability measure $ \D\mu=E\D s/\Phi$:\begin{align}\oint_{\Sigma_\varphi}\left(\frac{\kappa}{E}-\oint_{\Sigma_\varphi}\frac{\kappa}{E}\D \mu\right)^{2}\D \mu=\oint_{\Sigma_\varphi}\left\vert \bm n\times\nabla\frac{1}{E} \right\vert^{2}\D \mu.\label{eq:kE_fluct}\end{align}From this identity, it is also clear that  $ E(\bm r)\equiv\text{const.},\bm r\in\Sigma_\varphi$ implies that  $ \Sigma_\varphi$ is a circle.

\begin{proposition}[A geometric conservation law]We have $  \mathscr E'(\varphi)=0$ and  $ \mathscr E(\varphi)\equiv\frac{4\pi^{2}}{\Phi}$ in 2-exD.\end{proposition}\begin{proof}
We first enlist the help from \eqref{eq:phi_flow} to compute the derivative $  \mathscr E'(\varphi)$ as follows:
\begin{align}&\frac{\D}{\D\varphi}\left[\oint_{\Sigma_\varphi}\frac{\kappa^2\D s}{E}-\oint_{\Sigma_\varphi}\frac{\left\vert\bm n\times\nabla\log E\right\vert^{2}\D s}{E}\right]\notag\\={}&\oint_{\Sigma_\varphi}\frac{2\kappa}{E}\left( \Delta_{\Sigma_{\varphi}}\frac1E +\frac{\kappa^2}{E}\right)\D s-\oint_{\Sigma_\varphi}\frac{\kappa^{2}}{E}\frac{2\kappa}{E}\D s-2\oint_{\Sigma_\varphi}\frac{1}{gE}\frac{\partial\log E}{\partial u}{\frac{\partial(\kappa/E)}{\partial u}\D u}\notag\\={}&\oint_{\Sigma_\varphi}\frac{2\kappa}{E}\frac{\partial^{2}}{\partial s^{2}}\frac1E\D s+\oint_{\Sigma_\varphi}\frac{\partial(1/ E)}{\partial s}{\frac{\partial(2\kappa/E)}{\partial s}\D s=\oint_{\Sigma_\varphi}\frac{\partial}{\partial s}\left(\frac{2\kappa}{E}\frac{\partial}{\partial s}\frac1E\right)\D s=0}.\label{eq:E'}\end{align}Then, we  equate $ \mathscr E(\varphi)$ with $ \lim_{\varphi\to-\infty}\mathscr E(\varphi)$.

To evaluate the last limit, we need two observations. First, the asymptotic behavior in \eqref{eq:kappa_E_asympt} immediately leads us to $\lim_{\varphi\to-\infty}\oint_{\Sigma_\varphi}\frac{\kappa^2\D s}{E}=\frac{4\pi^{2}}{\Phi}$. Second, we have the following multipole expansion for a harmonic function\footnote{The expression $\log|\nabla U(x\bm e_x+ y\bm e_y)|=\R\log f'(x+iy)$ is the real part of a complex-analytic function, hence harmonic.  Here, one can construct  $ f(x+iy)= U(x\bm e_x+ y\bm e_y)+iV(x\bm e_x+ y\bm e_y)$ from $U$ and its conjugate harmonic function $V$. } \begin{align}
\log E(|\bm r|(\bm e_x\cos\theta+\bm e_y\sin\theta))=c_0-\log|\bm r|+\frac{c_1\cos\theta+s_1\sin\theta}{|\bm r|}+O\left( \frac{1}{|\bm r|^2} \right)
\end{align} for constants $c_0,c_1,s_1$, which enables us to estimate\begin{align}
\left\vert\bm n\times\nabla\log E(\bm r)\right|=O\left( \frac{1}{|\bm r|^2} \right),\quad \text{as }|\bm r|\to+\infty,\label{eq:dlogE_est}
\end{align} so $ \lim_{\varphi\to-\infty}\oint_{\Sigma_\varphi}\frac{\left\vert\bm n\times\nabla\log E\right\vert^{2}\D s}{E}=0$.

Summarizing what we have done in the last paragraph, we see that $ \mathscr E(\varphi)\equiv\frac{4\pi^{2}}{\Phi}$ holds.\end{proof}
\begin{corollary}[Sharp correlation comparison inequality]\label{cor:corr_comp}In 2-exD, we have $ \cov_{\Sigma_{\varphi}}\left( \frac{\kappa}{2\pi},\log \frac{EL_{\Sigma_{\varphi}}}{\Phi}\right)\geq\cov_{\Sigma_{\varphi}}\left( \frac{E}{\Phi},\log \frac{EL_{\Sigma_{\varphi}}}{\Phi}\right)$
according to  \eqref{eq:cov3}. The inequality is strict unless $ \partial\Omega$ is circular.
\end{corollary}
\begin{proof}We note that the derivative of \begin{align}&L_{\Sigma_\varphi}\left[
\cov_{\Sigma_{\varphi}}\left( \frac{\kappa}{2\pi},\log \frac{EL_{\Sigma_{\varphi}}}{\Phi}\right)-\cov_{\Sigma_{\varphi}}\left( \frac{E}{\Phi},\log \frac{EL_{\Sigma_{\varphi}}}{\Phi}\right)\right]\notag\\={}&\oint_{\Sigma_\varphi}\kappa\log \frac{EL_{\Sigma_\varphi}}{\Phi}\frac{\D s}{2\pi}-\oint_{\Sigma_\varphi}E\log \frac{EL_{\Sigma_\varphi}}{\Phi}\frac{\D s}{\Phi} ,\label{eq:k-E}
\end{align} evaluates to\begin{align}&
\frac{1}{2\pi}\oint_{\Sigma_\varphi}\left(\Delta_{\Sigma}\frac{1}{E}\right)\log \frac{EL_{\Sigma_\varphi}}{\Phi}\D s+\frac{1}{2\pi}\oint_{\Sigma_\varphi}\frac{\kappa^2\D s}{E}-\frac{1}{\Phi}\oint_{\Sigma_\varphi}\kappa\D s\notag\\={}&\frac{1}{2\pi}\left(\oint_{\Sigma_\varphi}\frac{\kappa^2\D s}{E}+ \oint_{\Sigma_\varphi}\frac{|\bm n\times\nabla\log E|^2\D s}{E}-\frac{4\pi^{2}}{\Phi}\right)=\frac{1}{\pi}\oint_{\Sigma_\varphi}\frac{|\bm n\times\nabla\log E|^2\D s}{E}\geq0,\label{eq:k'-E'}
\end{align}where the conservation law  \eqref{eq:Ephi} has entered the penultimate step. Now that the expression in \eqref{eq:k-E} vanishes as $ \varphi\to-\infty$, by virtue of  \eqref{eq:kappa_E_asympt}, we have the correlation comparison inequality stated in \eqref{eq:cov3}.

We note that  \eqref{eq:cov3} becomes an equality only when there is constant field intensity on each equipotential curve: $E(\bm r)\equiv\text{const.},\bm r\in\Sigma_\varphi$. By \eqref{eq:kE_fluct}, we know that this is equivalent to the requirement that
$\partial \Omega$ be a circle.\end{proof}

\subsection{Statistical (mis)alignment of tangential gradients\label{subsec:stat_align}}We say that the tangential gradients of two scalar functions $ f_1(\bm r),\bm r\in\Sigma_\varphi$ and $ f_2(\bm r),\bm r\in\Sigma_\varphi$ are aligned (resp.~misaligned) at a point $\bm r$
if \begin{align}
[\bm n\times\nabla f_1(\bm r)]\cdot[\bm n\times\nabla f_2(\bm r)]\geq0\quad(\text{resp. }[\bm n\times\nabla f_1(\bm r)]\cdot[\bm n\times\nabla f_2(\bm r)]\leq0).
\end{align}We say that there are statistical alignment (resp.~misalignment) of tangential gradients if \begin{align}
\oint_{\Sigma_\varphi}[\bm n\times\nabla f_1(\bm r)]\cdot[\bm n\times\nabla f_2(\bm r)]\D s\geq0\quad\left(\text{resp. }\oint_{\Sigma_\varphi}[\bm n\times\nabla f_1(\bm r)]\cdot[\bm n\times\nabla f_2(\bm r)]\D s\leq0\right).
\end{align}Statistical (mis)alignment  of tangential gradients tells us how  the monotonicity of two functions, on average, are tied to each other.

Before moving onto the convexity proof for \begin{align}
\mathscr F(\varphi):={}&\oint_{\Sigma_\varphi}\frac{\left\vert\bm n\times\nabla\log E\right\vert^{2}\D s}{E}=\oint_{\Sigma_\varphi}\frac{\kappa^2\D s}{E}-\frac{4\pi^{2}}{\Phi},
\end{align}we need a convenient formula for second-order derivatives in $\varphi$.
\begin{lemma}For any suitably regular $f(\bm r)$ defined in a neighborhood of $ \Sigma_\varphi$, we have \begin{align}
\frac{\D^2}{\D \varphi^2}\oint_{\Sigma_\varphi}E(\bm r)f(\bm r)\D s=\oint_{\Sigma_\varphi}\frac{1}{E}\Delta f\D s.\label{eq:Delta''}
\end{align}\end{lemma}\begin{proof}
Exploiting the decomposition of Laplacian in \eqref{eq:Lap_Lap}, we  can show that\begin{align} \frac{\D^2}{\D \varphi^2}\oint_{\Sigma_\varphi}E(\bm r)f(\bm r)\D s=\oint _{\Sigma_\varphi}\frac1E(\Delta f-\Delta _{\Sigma }f)\D s-\oint _{\Sigma_\varphi}\frac{1}{g}\frac{\partial (1/E)}{\partial u}\frac{\partial f }{\partial u}\D s.\end{align}Integrating by parts, we may further deduce\begin{align}-
\oint _{\Sigma_\varphi}\frac1E\Delta _{\Sigma }f\D s-\oint _{\Sigma_\varphi}\frac{1}{g}\frac{\partial (1/E)}{\partial u}\frac{\partial f }{\partial u}\D s=-
\oint _{\Sigma_\varphi}\frac1E\frac{\partial ^{2}f}{\partial s^{2}}\D s-\oint _{\Sigma_\varphi}\frac{\partial (1/E)}{\partial s}\frac{\partial f }{\partial s}\D s=0,
\end{align}hence our claim.\end{proof}
\begin{proposition}[Convexity of $ \mathscr F(\varphi)$]We have $\mathscr F''(\varphi)\geq0 $ in 2-exD, where the equality holds if and only if  $ \nabla\frac{\kappa}{E}=\mathbf 0$  on $ \Sigma_\varphi$.\end{proposition}\begin{proof}
By virtue of \eqref{eq:Delta''} and Talenti's relation $ \Delta\frac{\kappa}{E}=0$ \cite{Talenti1983}, we can quickly compute\begin{align}\begin{split}\mathscr F''(\varphi)={}&
\frac{\D^2}{\D\varphi^2}\oint_{\Sigma_\varphi}\frac{\kappa^2\D s}{E}=\oint_{\Sigma_\varphi}\frac{1}{E}\Delta\frac{\kappa^2}{E^{2}}\D s\\={}&2\oint_{\Sigma_\varphi}\left\vert\nabla\frac{\kappa}{E}\right\vert^2\D s+2\oint_{\Sigma_\varphi}\frac{1}{E}\frac{\kappa}{E}\Delta\frac{\kappa}{E}\D s=2\oint_{\Sigma_\varphi}\frac{1}{E}\left\vert\nabla\frac{\kappa}{E}\right\vert^2\D s\geq0,\end{split}\label{eq:F''}
\end{align}as stated. \end{proof}\begin{proposition}[Sharp inequality for $ \mathscr F'(\varphi)$]In 2-exD, we have $ \mathscr F'(\varphi)\geq0$, which entails  \eqref{eq:grad_prod}: $\left\langle \left[\bm n\times\nabla\frac{\kappa(\bm r)}{E(\bm r)}\right]\cdot\left[\bm n\times\nabla\frac{1}{E(\bm r)} \right] \right\rangle_{\Sigma}\leq0$. Both inequalities are strict unless $ \partial\Omega$ is a circle. \end{proposition}\begin{proof}From the convexity $ \mathscr F''(\varphi)\geq0$ we can deduce the monotonicity of $ \mathscr F(\varphi)$:\begin{align}\begin{split}0\geq-
\mathscr F'(\varphi)={}&-\frac{\D}{\D\varphi}\oint_{\Sigma_\varphi}\frac{\kappa^2\D s}{E}=-2\oint_{\Sigma_\varphi}\frac{\kappa}{E}\frac{\partial}{\partial \varphi}\left(\frac{\kappa}{E}\right)E\D s\\={}&-2\oint_{\Sigma_\varphi}\frac{\kappa}{E}\frac{\partial^{2}}{\partial s^{2}}\frac{1}{E}\D s=2\oint_{\Sigma_\varphi}\left[\bm n\times\nabla\frac{\kappa}{E}\right]\cdot\left[\bm n\times\nabla\frac{1}{E} \right]\D s,\end{split}
\end{align}after we establish $ \lim_{\varphi\to-\infty}
\mathscr F'(\varphi)=0$ on the asymptotic behavior:\begin{align}
\left|\bm n\times\nabla\frac{\kappa(\bm r)}{E(\bm r)}\right\vert=O\left( \frac{1}{|\bm r|^2} \right),\quad \left\vert \bm n\times\nabla\frac{1}{E(\bm r)} \right\vert=O(1),
\end{align}as $ |\bm r|\to+\infty$. Here, the tangential derivative of the harmonic function $ \frac\kappa E$ can be estimated in a similar fashion as \eqref{eq:dlogE_est}.

  This proves the statistical (mis)alignment of tangential gradients stated in \eqref{eq:grad_prod}.
This is a strict inequality unless $ \nabla\frac{\kappa}{E}=0$ on all the level sets enclosing $ \Sigma_\varphi$, in view of \eqref{eq:F''}. Since $0=\frac{\partial}{\partial \varphi}\frac{\kappa}{E}=\frac{1}{E}\frac{\partial^{2}}{\partial s^{2}}\frac{1}{E} $ implies constant $E(\bm r)$ on each level set, and $\bm n \times \nabla\frac{\kappa}{E}=\mathbf 0$ further entails constant $ \kappa(\bm r)$ on each level set, the situation  $ \mathscr F'(\varphi)=0$ happens only if $ \partial \Omega$ is a circle.\end{proof}
\begin{remark}
Admittedly, at this point, we have not yet exhausted all the possible geometric integrals that are convex in $\varphi$. By  \eqref{eq:Delta''} and the Kong--Xu equation $ \Delta\frac{\bm n\times\nabla\kappa}{E^{2}}=0$ \cite{KongXu2015}, one can also show that $
\frac{\D^{2} }{\D \varphi^{2}}\oint_{\Sigma_\varphi}E^{-3}|\bm n\times\nabla \kappa|^2\D s\geq0$. However, we are unable to reinterpret the first-order derivative $ \frac{\D }{\D \varphi}\oint_{\Sigma_\varphi}E^{-3}|\bm n\times\nabla \kappa|^2\D s$  as statistical (mis)alignment of geometrically/physically interesting quantities, as in the case of $
\mathscr F'(\varphi) $.\eor\end{remark}

\subsection{Longinetti functional and weighted correlation\label{subsec:Lfunc}}

In \cite[Theorem 4.1]{Longinetti1988}, Longinetti has shown
that $ \mathscr L(\varphi):=\log\oint_{\Sigma_{\varphi}}\frac{\D s}{E}$ is convex in $\varphi$, using  support functions in  convex domains.
 One can generalize his result to non-convex domains.
\begin{proposition}[Convexity of $\mathscr L(\varphi)$]
 In 2-exD, we have   $ \mathscr L''(\varphi)\geq0$, namely\begin{align}\oint_{\Sigma_{\varphi}}\frac{\D s}{E}
\frac{\D^{2} }{\D\varphi^{2}}\oint_{\Sigma_{\varphi}}\frac{\D s}{E} \geq\left(\frac{\D }{\D\varphi}\oint_{\Sigma_{\varphi}}\frac{\D s}{E}\right)^2.\label{eq:L''_nonneg}
\end{align}The equality holds only when $\partial\Omega$ is a circle.\end{proposition}\begin{proof}
Since the differential formulae in  \eqref{eq:phi_flow} and \eqref{eq:L'phi} bring us\begin{align}
\frac{\D}{\D\varphi}\oint_{\Sigma_\varphi}\frac{\D s}{E}=-2\oint_{\Sigma_\varphi}\frac{\kappa\D s}{E^{2}},
\end{align}and \begin{align}
\frac{\D}{\D\varphi}\oint_{\Sigma_\varphi}\frac{\D s}{E}=-2\oint_{\Sigma_\varphi}\frac{1}{E^{2}}\frac{\partial^{2}}{\partial s^{2}}\frac{1}{E}\D s+4\oint_{\Sigma_\varphi}\frac{\kappa^{2}\D s}{E^{3}}=4\oint_{\Sigma_\varphi}\left( \left\vert \bm n\times\nabla\frac{1}{E} \right\vert^{2}+\frac{\kappa^{2}}{E^{2}} \right)\frac{\D s}{E},
\end{align}we may proceed with the computation\begin{align}&
\oint_{\Sigma_{\varphi}}\frac{\D s}{E}
\frac{\D^{2} }{\D\varphi^{2}}\oint_{\Sigma_{\varphi}}\frac{\D s}{E} -\left(\frac{\D }{\D\varphi}\oint_{\Sigma_{\varphi}}\frac{\D s}{E}\right)^2\notag\\={}&4\left( \oint_{\Sigma_\varphi}\frac{\D s}{E} \right)^2\oint_{\Sigma_\varphi}\left[\left\vert \bm n\times\nabla\frac{1}{E} \right\vert^{2}+\left(\frac{\kappa}{E}-\oint_{\Sigma_\varphi}\frac{\kappa}{E}\D \lambda\right)^{2}\right]\D \lambda\geq0,
\end{align}for a probability measure $ \D \lambda=\frac{\D s}{E}\big/\oint_{\Sigma_\varphi}\frac{\D s}{E}$. So far, we have proved Longinetti's inequality in  \eqref{eq:L''_nonneg}, for all 2-exD problems whose boundary $ \partial \Omega$ is a smooth Jordan curve. Arguing as the proof of Corollary  \ref{cor:corr_comp}, we know that  $ \mathscr L''(\varphi)=0$ happens only when  $ \partial \Omega$ is a circle.\end{proof}

From $    \mathscr L''(\varphi)\geq0 $, we may deduce the inequality stated in \eqref{eq:Longinetti_cor}:\begin{align}
 \mathscr L'(\varphi)=\frac{-2\oint_{\Sigma_\varphi}\frac{\kappa\D s}{E^{2}}}{\oint_{\Sigma_\varphi}\frac{\D s}{E}}\geq\lim_{\phi\to-\infty}\mathscr L'(\phi)=-\frac{4\pi}{\Phi}.\label{eq:L'}
\end{align}This may also be rewritten as a non-negative correlation between $ E^2$ and $ \kappa/E$, weighted by  the probability measure $ \D \lambda=\frac{\D s}{E}\big/\oint_{\Sigma_\varphi}\frac{\D s}{E}$:\begin{align}
\frac{2\pi}{\oint_{\Sigma_\varphi}\frac{\D s}{E}}=\oint_{\Sigma_{\varphi}}E^2\frac{\kappa}{E}\D\lambda\geq\oint_{\Sigma_{\varphi}}E^{2}\D\lambda\oint_{\Sigma_{\varphi}}\frac{\kappa}{E}\D\lambda=\frac{\Phi\oint_{\Sigma_\varphi}\frac{\kappa\D s}{E^{2}}}{\left(\oint_{\Sigma_\varphi}\frac{\D s}{E}\right)^2}.
\end{align}

 \section{Geometric relations for Green's functions\label{sec:Green}}
For the level sets $ \Sigma_\varphi$ of Green's functions in 2-inD, our arguments in \S\ref{sec:ent} still bring us $ \mathscr H''(\varphi)\geq0$, $\mathscr E'(\varphi)=0$ and $ \mathscr F''(\varphi)\geq0$. However, the asymptotic analysis  in 2-inD ($|\bm r|\to 0,\varphi\to+\infty$) is critically different from that in 2-exD ($|\bm r|\to +\infty,\varphi\to-\infty$).

To prove that $ \lim_{\varphi\to+\infty}\mathscr E(\varphi)=\frac{4\pi^2}{\Phi}$, we need two steps. First, we note that the asymptotic behavior in \eqref{eq:kappa_E_asympt} also applies to $|\bm r|\to0$, so $\lim_{\varphi\to+\infty}\oint_{\Sigma_\varphi}\frac{\kappa^2\D s}{E}=\frac{4\pi^{2}}{\Phi}$. Second, in view of the multipole expansion \begin{align}
\log E(|\bm r|(\bm e_x\cos\theta+\bm e_y\sin\theta))=-\log|\bm r|+|\bm r|(c_1'\cos\theta+s_1'\sin\theta)+O( |\bm r|^2)
\end{align}as  $|\bm r|\to0$, we have $ |\bm n\times\nabla\log E(\bm r)|=O(1)$ and $ \frac{\D s}{E(\bm r) }=O( |\bm r|^2\D \theta)$, so $ \lim_{\varphi\to+\infty}\oint_{\Sigma_\varphi}\frac{|\bm n\times\nabla\log E|^2\D s}{E}=0$.

To prove \eqref{eq:cov3'}, one simply checks that \eqref{eq:k-E} and \eqref{eq:k'-E'} remain valid in 2-inD:\begin{align}\begin{split}{}&\frac{\D}{\D\varphi}\left\{
L_{\Sigma_\varphi}\left[
\cov_{\Sigma_{\varphi}}\left( \frac{\kappa}{2\pi},\log \frac{EL_{\Sigma_{\varphi}}}{\Phi}\right)-\cov_{\Sigma_{\varphi}}\left( \frac{E}{\Phi},\log \frac{EL_{\Sigma_{\varphi}}}{\Phi}\right)\right]\right\}\\={}&\frac{1}{\pi}\oint_{\Sigma_\varphi}\frac{|\bm n\times\nabla\log E|^2\D s}{E}\geq0,\end{split}
\end{align}and that \begin{align}
\lim_{\varphi\to+\infty}L_{\Sigma_\varphi}\left[
\cov_{\Sigma_{\varphi}}\left( \frac{\kappa}{2\pi},\log \frac{EL_{\Sigma_{\varphi}}}{\Phi}\right)-\cov_{\Sigma_{\varphi}}\left( \frac{E}{\Phi},\log \frac{EL_{\Sigma_{\varphi}}}{\Phi}\right)\right]=0.
\end{align}

One can then compute directly that $ \lim_{\varphi\to+\infty}\mathscr H'(\varphi)=0$ and $ \lim_{\varphi\to+\infty}
\mathscr F'(\varphi)=0$, so we have $\mathscr H'(\varphi)\leq0$ and $ \mathscr F'(\varphi)\leq0$ for 2-inD, thus confirming \eqref{eq:cov1'} and \eqref{eq:grad_prod'}. To prove \eqref{eq:Longinetti_cor'}, it would suffice to write down $\mathscr  L'(\varphi)\leq\lim_{\phi\to+\infty}\mathscr L'(\phi)=-\frac{4\pi}{\Phi}$, by analogy to  \eqref{eq:L'}.

The equalities in \eqref{eq:cov1'},  \eqref{eq:cov3'}, \eqref{eq:grad_prod'} and  \eqref{eq:Longinetti_cor'} hold only when $ E$ and $\kappa $ both remain constant on each level set, which means that $ \partial \Omega$ is a circle centered at $ \mathbf 0$.

One may further adapt our results for 2-inD to annular domains with the Bernoulli boundary condition imposed on the inner ring.

Before closing this section, we point out that the reversal of inequality signs in 2-inD (as compared to 2-exD)   is not physically unexpected.
As time elapses in  internal diffusion limited aggregation \cite{MeakinDeutch1986,DiaconisFulton1991}, or the corresponding Hele-Shaw flow as its deterministic limit \cite{LevinePeres2010},  the interface becomes more and more circular \cite{LawlerBramsonGriffeath1992,JerisonLevineSheffield2012}, rather than more and more spiky as in   the Witten--Sander process \cite{DLA-PRL,DLA-PRB}. This indicates that the curvature feedback mechanism in 2-inD runs  opposite to its counterpart in 2-exD. \section{Discussions and outlook\label{sec:misc}}

It is appropriate to  consider  Dirichlet problems for harmonic functions in  higher dimensional Euclidean spaces  $ \mathbb R^d$ ($d>2$). One can formulate $ d$-exD as a Laplace equation \begin{align}
\nabla^2 U(\bm r)=0,\quad \bm r\in\Omega\subset \mathbb R^d\tag{\ref{eq:2D_Laplace}$'$}
\end{align}in an unbounded domain $ \Omega$, whose boundary $ \partial \Omega$ is a smooth and connected (hyper)surface, on which $U(\bm r) $ remains constant.  Such a boundary is orientable, with an outward unit normal vector  $\bm n$ defined everywhere on $ \partial \Omega$. The flux condition \begin{align}
-\oint_{\partial\Omega}\bm n\cdot\nabla U(\bm r)\D \Sigma=\Phi>0\tag{\ref{eq:2D_flux}$'$}
\end{align}  (with $\D \Sigma$ being the induced Lebesgue measure on $ \partial\Omega$) translates into  the following asymptotic behavior as  $ |\bm r|$ goes to infinity:\begin{align}
 U(\bm r)\sim \frac{\Phi }{4\pi^{d/2}|\bm r|^{d-2}}\int_0^\infty t^{(d-4)/2}e^{-t}\D t .\tag{\ref{eq:U_inf}$'$}
\end{align}   If $ \mathbf 0\notin\Omega\cup\partial\Omega$, then one can  define the Green's function in  $ d$-inD as a solution to \begin{align}
\left\{\begin{array}{ll}
\nabla^2 G_{D}^{\partial\Omega}(\mathbf 0,\bm r)=0, & \bm r\in\mathbb R^d\smallsetminus(\Omega\cup\partial\Omega\cup\{\mathbf 0\}),\\
G_{D}^{\partial\Omega}(\mathbf 0,\bm r)=\text{const.,} & \bm r\in \partial\Omega, \\\multicolumn{2}{l}{\displaystyle-\lim_{\varepsilon\to0^+}\oint_{|\bm r|=\varepsilon}\bm n\cdot\nabla G_{D}^{\partial\Omega}(\mathbf 0,\bm r)\D \Sigma=1.}
\end{array}\right.\tag{\ref{eq:2-inD_G}$'$}
\end{align} In both   $ d$-exD  and   $ d$-inD, we will be interested in the geometric relations on the level sets $ \Sigma_\varphi$ [equipotential (hyper)surfaces in physical parlance], in a similar vein as \S\S\ref{sec:ent}--\ref{sec:Green}.

The possible existence of critical points (where the gradient of a harmonic function vanishes) forms a major obstacle to finding higher dimensional analogs of  the monotonicity results in 2-exD (\S\ref{sec:ent}) and 2-inD (\S\ref{sec:Green}). Lacking a generalization of  the Riemann mapping theorem to $ \mathbb R^d$ ($d>2$), we usually cannot rule out the critical points, by relying on the smoothness and connectedness of the boundary $ \partial \Omega$ alone.
A recent result of Ma--Zhang \cite[Proposition 3.2]{MaZhang2014} ensures the non-existence of critical points in $ d$-exD and $d$-inD if $ \partial \Omega$ is both smooth and convex. Thus, it is sensible to limit our scope to $ d$-exDc and $d$-inDc, which are problems with the additional constraint on geometric convexity of the boundary $ \partial \Omega$. (For the special case of 3-inDc, one can also deduce the non-existence of critical points from Gergen's answer \cite[(1.1)]{Gergen1931} to a question of Morse.)

Yet another obstacle to obtaining higher dimensional generalizations of the geometric inequalities in \S\S\ref{sec:ent}--\ref{sec:Green} is the unavailability of auxiliary harmonic functions, like those studied by Talenti \cite{Talenti1983} and Kong--Xu \cite{KongXu2015}. Without such ``extra harmonicity'' in higher dimensional spaces, we may lose control of the sign in certain derivatives with respect to $\varphi$.

Despite these difficulties, we can still construct some (partial and conditional) generalizations to the current work in $ \mathbb R^d$ ($d>2$). For example, instead of a conservation law in \S\ref{subsec:cons_law}, our closest analogs in $ \mathbb R^3$ \cite[Theorem 1.1]{3DconvEPS} will be the following inequality for every level set $ \Sigma_\varphi$ in 3-exDc (strict unless $ \partial\Omega $ is a sphere):\begin{align}
\oint_{\Sigma_\varphi}\frac{4[H^2(\bm r)-K(\bm r)]-|\bm n\times\nabla\log |\nabla U(\bm r)| |^{2}}{|\nabla U(\bm r)|}\D S\geq0,\label{eq:3-exDc}
\end{align} and following inequality for every level set $ \Sigma_\varphi$ in 3-inDc (strict unless $ \partial\Omega $ is a sphere centered at the origin):\begin{align}
\oint_{\Sigma_\varphi}\frac{4[H^2(\bm r)-K(\bm r)]-|\bm n\times\nabla\log |\nabla G_{D}^{\partial\Omega}(\mathbf 0,\bm r)| |^{2}}{|\nabla G_{D}^{\partial\Omega}(\mathbf 0,\bm r)|}\D S\leq0.\label{eq:3-inDc}
\end{align}Here, the mean curvature and the Gaussian curvature are denoted by   $H$ and $K $ respectively, while $ \D S$ is the surface measure.
\subsection*{Acknowledgments}Part of this work was assembled from  my research notes in 2006 (on curvature effects in nanophotonics) and 2011 (on entropy in curved spaces). I am grateful to Prof.\ Xiaowei Zhuang (Harvard) and Prof.\ Weinan E (Princeton) for their thought-provoking questions in 2006 and 2011 that inspired these research notes. I  thank Prof.\ David Jerison (MIT) for his comments on \eqref{eq:cov2} in 2011. I am indebted to an anonymous reviewer, whose suggestions helped improve the presentation of this paper.


\end{document}